%% file: sign-en.tex
\documentclass[11pt]{amsart}

\usepackage{amsmath, amsthm, amssymb, amsfonts}
\usepackage{vmargin}
\usepackage[latin1]{inputenc}
\usepackage{graphicx,color}
\usepackage{hyperref}

\usepackage[all]{xy}

\numberwithin{equation}{section}

\newcommand{\zzz}{\mathbb{Z}}

\newtheorem{theorem}{Theorem}[section]
\newtheorem{proposition}[theorem] {Proposition}
\newtheorem{lemma}[theorem]{Lemma}

\theoremstyle{definition}
\newtheorem{definition}[theorem]{Definition}
\theoremstyle{remark}

\theoremstyle{remark} 
\newtheorem*{rembis}{Remark}
\theoremstyle{remark}

\newcommand\Xs{\mathbb{X}}
\newcommand\Os{\mathbb{O}}
\newcommand\Pent{\mathrm{Pent}}
\newcommand\Hex{\mathrm{Hex}}
\newcommand\Rect{\mathrm{Rect}}

\def\Caa{\widetilde{C}}

\def\Os{\mathbb O}

\def\ss {{\mathfrak{S}}}

\def\fin\qedhere
\def\pr {{\text{pr}}}

\def\un{\mathbf 1 }
\def\wun{{\widetilde{ \mathbf{1} }}}

\def\x{\mathbf x}
\def\y{\mathbf y}
\def\z{\mathbf z}

\def\wx{\widetilde{\mathbf x}}
\def\wy{\widetilde{\mathbf y}}
\def\wz{\widetilde{\mathbf z}}
\def\wss{\widetilde{\mathfrak S}}

\def\S{\mathbf S}
\def\alphas{\mathbb\alpha}
\def\betas{\mathbb\beta}

\newcommand\EmptyRect{\Rect^o}

\def\EmptyRect{\Rect^\circ}

\def\wdm{\widetilde{\partial}^-}
\newcommand\wtau{\widetilde{\tau}}
\def\wr{\widetilde{r}}

\def\wp{\widetilde{p}}
\def\wcm{\widetilde{C}^-}

\begin{document}
\title[Sign refinement]{Sign refinement for combinatorial\\link Floer homology}
\author[\'Etienne Gallais]{\'Etienne Gallais}
\address {LMAM - Universit\'e de Bretagne Sud, BP 573, 56017 VANNES, FRANCE}
\email {etienne.gallais@univ-ubs.fr}

\begin {abstract} 
Link Floer homology is an invariant for links which has recently been described entirely in a combinatorial way.
Originally constructed with mod 2 coefficients, it was generalized to integer coefficients thanks to a sign refinement.
In this paper, thanks to the spin extension of the permutation group we give an alternative construction of the combinatorial link Floer chain complex associated to a grid diagram with integer coefficients.
We prove that the filtered homology of this complex is an invariant for the link and that it gives the previous sign refinement by means of a $2$-cohomological class corresponding to the spin extension of the permutation group.
\end {abstract}
\today
\maketitle

\input{introduction}
\input{algebraic-preliminaries}
\input{definition-complex}
\input{properties}
\input{link-sign-refinement}

\end{document}

%% file: introduction.tex
\section{Introduction}
Heegaard-Floer homology \cite{OS1} is an invariant for closed oriented 3-manifolds which was extended to give an invariant for null-homologous oriented links in such manifolds called link Floer homology \cite{OS2}, \cite{OS3} \cite{Rasmussen}.
It gives the Seifert genus $g(K)$ of a knot $K$ \cite{OS4}, detects fibered knots (\cite{ghiggini} in the case where $g(K)=1$ and \cite{ni} in general) and its graded Euler characteristic gives the Alexander polynomial (\cite{OS2}, \cite{Rasmussen}).
Recently, a combinatorial description of link Floer homology was given \cite{MOS} and its topological invariance was proved in a purely combinatorial way \cite{MOST}.
The purpose of this paper is to give an alternative description of combinatorial link Floer homology with $\zzz$ coefficients.
This point of view was recently used by Audoux \cite{Aud07} to describe combinatorial Heegaard--Floer homology for singular knots.

Let first recall the context of combinatorial link Floer homology: we follow conventions of \cite{MOST}.
A planar grid diagram $G$ lies in a square on the plane with $n\times n$ squares where $n$ is the complexity of $G$.
Each square is decorated with an $X$, an $O$ or nothing in such a way that each row and each column contains exactly one $X$ and one $O$.
We number the $X$'s and the $O$'s from 1 to $n$ and denote $\Xs$ the set $\{X_i \}_{i=1}^n $ and $\Os$ the set $\{O_i \}_{i=1}^n$.

Given a grid diagram $G$, we place it in standard position on the plane as follows: the bottom left corner is at the origin and each cell is a square of length one.
We construct a planar link projection by drawing horizontal segments from the $O$'s to the $X$'s in each row and vertical segments from the $X$'s to the $O$'s in each column.
At each intersection point, the vertical segment is over the horizontal one.
This gives an oriented link $\overrightarrow L$ in $S^3$ and we say that $\overrightarrow L$ has a grid presentation given by $G$.

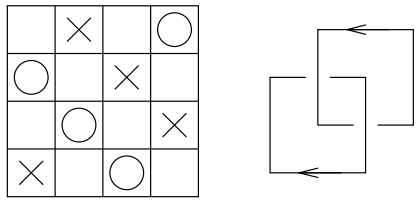
\begin{figure}[!ht]
  \begin{center}
	\input{hopf.pstex_t} 
  \end{center}
   \caption{\footnotesize \textbf{Grid presentation of the Hopf link.} }
   \label{fig:diagramme-en-grille-hopf}
\end{figure}

We place the grid diagram on the oriented torus $\mathcal T$ by making the usual identification of the boundary of the square.
We endow $\mathcal T$ with the orientation induced by the planar orientation.
Let $\alphas$ be the collection of the horizontal circles and $\betas$ the collection of the vertical ones.
We associate with $G$ a chain complex $(C^-,\partial^-)$: it is the group ring of $\ss_n$ over $\zzz/2\zzz[U_{O_1},\ldots,U_{O_n}]$ where $\ss_n$ is the permutation group of $n$ elements.
A generator $\x\in\ss_n$ is given on $G$ by its graph: we place dots in points $(i,x(i))$ for $i=0,\ldots,n-1$ (thus the fundamental domain of $G$ is the square minus the right vertical segment and the top horizontal segment).

For $A,B$ two finite sets of points in the plane we define $\mathcal{I}(A,B)$ to be the number of pairs $(a_1,a_2)\in A$ and $(b_1,b_2)\in B$ such that $a_1 < b_1$ and $a_2 < b_2$.
Let $\mathcal{J}(A,B) = (\mathcal{I}(A,B) + \mathcal{I}(B,A))/2$.
We provide the set of generators with a Maslov degree $M$ given by
$$
M(\x) =\mathcal{J}(\x-\Os,\x-\Os)+1
$$
where we extend $\mathcal J$ by bilinearly over formal sums (or differences) of subsets.
Each variable $U_{O_i}$ has a Maslov degree equal to $-2$ and constants have Maslov degree equal to zero.
Let $M_S(\x)$ be the same as $M(\x)$ with the set $S$ playing the role of $\Os$.

We provide the set of generators with an Alexander filtration $A$ given by $A(\x) = (A_1(\x),\ldots,A_l(\x))$ with
$$
A_i(\x) = \mathcal{J}(\x -\dfrac{1}{2}(\Xs + \Os) , \Xs_i - \Os_i ) - \dfrac{n_i - 1}{2}
$$
where when we number the components of $\overrightarrow L$ from 1 to $l$, $\Os_i \subset \Os$ (respect. $\Xs_i \subset \Xs$) is the subset of $\Os$ (resp. $\Xs$) wich belongs to the $i$-th component of $\overrightarrow L$ and $n_i$ is the number of horizontal segments which belongs to the $i$-th component.
We let $A(U_{O_j}) = (0,\ldots,-1,0,\ldots,0)$ where $-1$ corresponds to the $i$-th coordonate if $O_j$ belongs to the $i$-th component of $\overrightarrow L$.

Given two generators $\x$ and $\y$ and an immersed rectangle $r$ in the torus whose edges are arcs in the horizontal and vertical circles, we say that $r$ connects $\x$ to $\y$ if $\y . \x^{-1}$ is a transposition, if all four corners of $r$ are intersection points in $\x \cup \y$, and if we traverse each horizontal boundary component of $r$ in the direction dictated by the orientation of $r$ induced by $\mathcal T$, then the arc is oriented from a point in $\x$ to the point in $\y$.
Let $\Rect(\x,\y)$ be the set of rectangles connecting $\x$ to $\y$: either it is the empty set or it consists of exactly two rectangles.
Here a rectangle $r\in \Rect(\x,\y)$ is said to be empty if there is no point of $\x$ in its interior.
Let $\Rect^{\circ}(\x,\y)$ be the set of empty rectangles connecting $\x$ to $\y$.

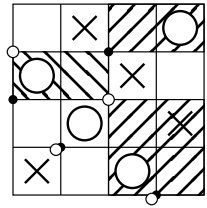
\begin{figure}[!ht]
  \begin{center}
	\input{rectangle.pstex_t} 
  \end{center}
   \caption{\footnotesize \textbf{Rectangles.} We mark with black dots the generator $\x$ and with white dots the generator $\y$.
There are two rectangles in $\Rect(\x,\y)$ but only the left one is in $\Rect^{\circ}(\x,\y)$.}
   \label{fig:rectangle}
\end{figure}

The differential $\partial^- : C^-(G) \rightarrow C^-(G)$ is given on the set of generators by
$$
\partial^- \x = \sum_{\y\in \ss_n} \sum_{r\in \Rect^{\circ}(\x,\y)} U_{O_1}^{O_1 (r)} \ldots U_{O_n}^{O_n (r)} . \y
$$
where $O_i(r)$ is the number of times $O_i$ appears in the interior of $r$.

\begin{theorem}[Manolescu-Ozsv\`ath-Sarkar \cite{MOS}]
  $(C^-(G),\partial^-)$ is a chain complex for $CF^-(S^3)$ with homological degree induced by $M$ and filtration level induced by $A$ which coincides with the link filtration of $CF^-(S^3)$.
\end{theorem}

In \cite{MOST}, the authors define a sign assigment for empty rectangles $\S:\Rect^{\circ} \rightarrow \{\pm1\}$.
Then, by considering $C^-(G)$ the group ring of $\ss_n$ over $\zzz[U_{O_1},\ldots,U_{O_n}]$ and the differential $\partial^-: C^-(G) \rightarrow C^-(G)$ given by
$$
\partial^- \x = \sum_{\y\in \ss_n} \sum_{r\in \Rect^{\circ}(\x,\y)} \S(r).U_{O_1}^{O_1 (r)} \ldots U_{O_n}^{O_n (r)} . \y
$$
they obtain the following result:
\begin{theorem}[Manolescu-Ozsv\`ath-Szab\'o-D. Thurston \cite{MOST}]
  Let $\overrightarrow L$ be an oriented link with $l$ components.
We number the $\Os$ so that $O_1,\ldots,O_l$ correspond to the different components of $\overrightarrow L$.
Then the filtered quasi-isomorphism type of $(C^-(G),\partial^-)$ over $\zzz[U_{O_1},\ldots,U_{O_l}]$ is an invariant of the link.
\end{theorem}

In this paper, we give a way to refine the complex over $\zzz$ thanks to $\wss_n$ the spin extension of $\ss_n$ which is a non-trivial central extension of $\ss_n$ by $\zzz/2\zzz$.
In section \ref{section:algebraic-preliminaries} we define the spin extension $\wss_n$ and make some algebraic calculus.
Let $z$ be the unique non-trivial central of $\wss_n$ and $\Lambda = \zzz[U_{O_1},\ldots,U_{O_n}]$.
In section \ref{section:definition-complex} we define a filtered chain complex $(\wcm(G),\wdm)$ where $\wcm(G)$ is the quotient module of the free $\Lambda$-module with generating set $\wss_n$ by the submodule generated by $\{z+1\}$.
In section \ref{section:properties} we prove that the filtered homology of $(\wcm(G)\otimes \zzz/2\zzz,\wdm)$ coincides with the filtered homology of the chain complex $(C^-(G),\partial^-)$ with coefficients in $\zzz/2\zzz$.
Then, following the proof in \cite{MOST} we obtain

\begin{theorem}
  Let $\overrightarrow L$ be an oriented link with $l$ components.
We number the $\Os$ so that $O_1,\ldots,O_l$ correspond to the different components of $\overrightarrow L$.
Then the filtered quasi-isomorphism type of $(\wcm(G),\wdm)$ over $\zzz[U_{O_1},\ldots,U_{O_l}]$ is an invariant of the link.
\end{theorem}

Finally, in section \ref{section:sign-assignment}, we prove that our chain complex defines a sign assignment in the sense of \cite{MOST} and that $(\wcm(G),\wdm)$ is filtered quasi-isomorphic to $(C^-(G),\partial^-)$ with coefficients in $\zzz$.

%% file: hopf.pstex_t
\begin{picture}(0,0)%
\includegraphics{hopf.pstex}%
\end{picture}%
\setlength{\unitlength}{3947sp}%
\begingroup\makeatletter\ifx\SetFigFont\undefined%
\gdef\SetFigFont#1#2#3#4#5{%
  \reset@font\fontsize{#1}{#2pt}%
  \fontfamily{#3}\fontseries{#4}\fontshape{#5}%
  \selectfont}%
\fi\endgroup%
\begin{picture}(2574,1224)(589,-1273)
\end{picture}%

%% file: rectangle.pstex_t
\begin{picture}(0,0)%
\includegraphics{rectangle.pstex}%
\end{picture}%
\setlength{\unitlength}{3947sp}%
\begingroup\makeatletter\ifx\SetFigFont\undefined%
\gdef\SetFigFont#1#2#3#4#5{%
  \reset@font\fontsize{#1}{#2pt}%
  \fontfamily{#3}\fontseries{#4}\fontshape{#5}%
  \selectfont}%
\fi\endgroup%
\begin{picture}(1256,1280)(557,-1329)
\end{picture}%

%% file: algebraic-preliminaries.tex
\section{Algebraic preliminaries}\label{section:algebraic-preliminaries}
Let $\mathfrak{S}_n$ be the group of bijections of a set with $n$ elements numbered from $0$ to $n-1$.
It is given in terms of generators and relations where the set of generators is $\{\tau_i\}_{i=0}^{n-2} $ with $\tau_{i}$ the transposition which exchanges $i$ and $i+1$ and relations are
$$\tau_{i}^2=\mathbf{1} \qquad 0\leq i \leq n-2 $$
$$\tau_{i}. \tau_{j}=\tau_{j}. \tau_{i}  \qquad \vert i-j \vert >1, \quad 0 \leq i,j\leq n-2 $$
$$\tau_{i} . \tau_{i+1} . \tau_{i}=\tau_{i+1} . \tau_{i}. \tau_{i+1} \qquad 0\leq i \leq n-3  $$

\begin{proposition}
  The group given by generators and relations
$$
\begin{tabular}{cl}
$\wss_n = < \wtau_0,\ldots,\wtau_{n-2},z \vert$ & $z^2 = \widetilde{\mathbf{1}} , z\wtau_i = \wtau_i z , \wtau_i^2=z,\quad 0\leq i\leq n-2$; \\
& $\wtau_{i} . \wtau_{j}= z \wtau_{j} . \wtau_{i}  \qquad \vert i-j \vert >1, \quad 0 \leq i,j\leq n-2; $ \\
& $\wtau_{i} .\wtau_{i+1} . \wtau_{i}=\wtau_{i+1} . \wtau_{i}. \wtau_{i+1} \qquad 0\leq i \leq n-3 \;> $
\end{tabular}
$$
is a non-trivial central extension ($n\geq 4$) of $\ss_n$ by $\zzz/2\zzz$ called the spin extension of $\ss_n$.
\end{proposition}

\begin{rembis}
  The terminology \emph{spin extension} of $\ss_n$ is inspired by \cite{Bessenrodt}, \cite{dijkgraaf-1999}.
\end{rembis}

\begin{proof}
Let $p:\wss_n \rightarrow \ss_n$ be the morphism given on generators by $p(\wtau_i)=\tau_i$ and $p(z)=\un$: it is onto.
The central subgroup $H = \{\widetilde{\mathbf{1}},z\}\subset \wss_n$ is such that $H \subset \ker p$.
The group $\wss_n /H$ admits a presentation by generators and relations which coincides with the one for $\ss_n$ and so $\overline{p}:\wss_n /H \rightarrow \ss_n$ is an isomorphism.
In particular $\vert \wss_n\vert = 2n!$ and $\wss_n$ is the following central extension of $\ss_n$ by $\zzz/2\zzz$:
$$\xymatrix{
1 \ar[r] & \zzz/2\zzz \ar@{^{(}->}[r]^{i} & \wss_n \ar[r]^p & \ss_n \ar[r]  & 1
}$$
We prove that it is non-trivial.
Let $\mathbb{Q}_8$ be the subgroup of $\wss_n$ generated by $\wtau_0,\wtau_2,z$.
Then $\mathbb{Q}_8$ is isomorphic to the unit sphere in the space of quaternions intersected with the lattice $\zzz^4$ by a morphism $\Phi$ such that $\Phi(\wtau_0)=i$, $\Phi(\wtau_2)=j$, $\Phi(\wtau_0 . \wtau_2)=k$ and $\Phi(z)=-1$.
Therefore $\wss_n$ is non-trivial.
\end{proof}

For $i<j$, define
$$
\wtau_{i,j} = \wtau_i . \wtau_{i+1} . \,\ldots \,.\wtau_{j-2} . \wtau_{j-1} .\wtau_{j-2} . \, \ldots \,.\wtau_{i+1} . \wtau_i
$$
and $\wtau_{j,i} = z\wtau_{i,j}$.

Let $\varepsilon:\ss_n \rightarrow \{0,1\}$ be the signature morphism.
\begin{lemma}\label{lemma:conjugation}
  Let $\wx = \wtau_{i_1} .\wtau_{i_2} . \ldots . \wtau_{i_k}$ be an element in $\wss_n$ and $\x = p(\wx)\in \ss_n$.
Then for any $0 \leq i\neq j \leq n-1$
$$
\wx . \wtau_{i,j} . \wx^{-1} = z^{\varepsilon(\x)} \wtau_{\x(i),\x(j)}
$$
\end{lemma}

\begin{proof}
Since $\wx = \wtau_{i_1} .\wtau_{i_2} . \ldots . \wtau_{i_k}$, $\wx^{-1} = z^{\varepsilon(\x)}\wtau_{i_k} . \ldots . \wtau_{i_1}$.
We prove by induction on $k\geq 1$ that for any $i,j\in \{0,\ldots,n-1\}$ we have $\wx . \wtau_{i,j} . \wx^{-1} = z^{\varepsilon(\x)} \wtau_{\x(i),\x(j)}$.
\begin{itemize}
\item \textbf{Initialization.} Let $\wx = \wtau_{l}$ and $0\leq i<j \leq n-1$.
So $\wtau_{l}^{-1} = z  \wtau_{l}$ and $\varepsilon(\x) = 1$.
There are several cases.
\begin{itemize}
\item \textbf{Case 1: $l<i-1$ or $l>j$.} $\wx . \wtau_{i,j} . z\wx = z\tau_{i,j}$.
\item \textbf{Case 2: $l=i-1$.} $\wx . \wtau_{i,j} . z\wx = z \wtau_{i-1}.\wtau_{i,j} . \wtau_{i-1} =z \wtau_{i-1,j}$ by definition.
\item \textbf{Case 3: $l=i$.} $\wtau_i . \wtau_{i,j} . z\wtau_{i} = z \wtau_{i+1,j}$.
\item \textbf{Case 4: $i< l <j-1$.} We prove by induction on $l-i\geq 1$ for $i,j$ fixed that $ \wtau_l . \wtau_{i,j} . z\wtau_l = z \wtau_{\tau(i),\tau(j)}$.
For $l=i+1$ then we have
$$
\begin{array}{rcl}
  \wtau_{i+1} . \wtau_{i,j} . z\wtau_{i+1} & =  & z \wtau_i . \wtau_{i+1}.\wtau_i .\wtau_{i+2,j}.\wtau_i . \wtau_{i+1}.\wtau_i \\
& = & z \wtau_i . \wtau_{i+1} . \wtau_{i+2,j}.\wtau_{i+1}.\wtau_i \\
& = & z \wtau_{i,j}
\end{array}
$$
Suppose it is proved until rank $(l-1)-i$.
Then for $\wx = \wtau_{l}$ with $l<j-1$ we have
$$
\begin{array}{rcl}
  \wx . \wtau_{i} . z \wx  & = & z \wtau_{l}.\wtau_{i,j} . \wtau_{l}\\
& = & z (\wtau_i. \ldots . \wtau_{l-2}) . (\wtau_l . \wtau_{l-1}.\wtau_l).\wtau_{l-1,j}. (\wtau_l.\wtau_{l-1} . \wtau_l) .(\wtau_{l-2} . \ldots .\wtau_{i}) \\

& = & z (\wtau_i. \ldots . \wtau_{l-2}) .( \wtau_{l-1}.\wtau_l . \wtau_{l-1}).\wtau_{l-1,j}.(\wtau_{l-1}.\wtau_{l} . \wtau_{l-1}) .(\wtau_{l-2} . \ldots .\wtau_{i}) \\

& = & z(\wtau_i. \ldots . \wtau_{l-1} . \wtau_{l}) . \wtau_{l-1,j} . (\wtau_l . \wtau_{l-1} . \ldots . \wtau_i) \; \text{by induction}\\

& = & z(\wtau_i. \ldots . \wtau_{l-1}) . \wtau_{l,j}. ( \wtau_{l-1} . \ldots . \wtau_i)\; \text{by induction}\\

& = & z \wtau_{i,j} \; \text{by case 2}\\
\end{array}
$$
\item \textbf{Case 5: $l=j-1$.} $\wtau_{j-1} . \wtau_{i,j} . z\wtau_{j-1} = z (\wtau_i . \ldots .\wtau_{j-3}).\wtau_{j-1}.\wtau_{j-2} . \wtau_{j-1}.\wtau_{j-2}.\wtau_{j-1}.(\wtau_{j-3}.\ldots.\wtau_{i})= z \wtau_{i,j-1}$.

\item \textbf{Case 6: $l=j$.} $\wtau_{j} . \wtau_{i,j} . z\wtau_{j} = z (\wtau_i . \ldots .\wtau_{j-2}).\wtau_{j}.\wtau_{j-1} . \wtau_{j}.(\wtau_{j-2}.\ldots.\wtau_{i})= z \wtau_{i,j+1}$.
\end{itemize}

\item \textbf{Heredity.} Suppose the property is true until rank $k$.
Let $\wx = \wtau_{i_1} .\wtau_{i_2} . \ldots . \wtau_{i_k}$ and $\wtau_{i,j}$ be two elements in $\wss_n$.
Denote $\wy = \wtau_{i_2} . \ldots . \wtau_{i_k}$.
Then $\wx . \wtau_{i,j} . \wx^{-1} = \wtau_{i_1} . \wy . \wtau_{i,j}  . \wy^{-1} . z \wtau_{i_1}$.
By induction hypothesis,
$$
 \wy . \wtau_{i,j}  . \wy^{-1} = z^{\varepsilon(\y)} . \wtau_{\y(i),\y(j)}
$$
So, $\wx . \wtau_{i,j} . \wx^{-1} = \wtau_{i_1} . z^{\varepsilon(\y)} . \wtau_{\y(i),\y(j)} . z \wtau_{i_1}$.
By induction hypothesis one more time,
$$
\wx . \wtau_{i,j} . \wx^{-1} = z^{\varepsilon(\y)+1}\wtau_{\tau_{i_1}. \y(i),\tau_{i_1}. \y(j)} = z^{\varepsilon(\x)}.\wtau_{\x(i),\x(j)}
$$
\end{itemize}

\end{proof}

The group $\wss_n$ has another presentation in terms of generators and relations.
Take $\{z'\} \cup \{\wtau_{i,j}' \}_{i\neq j}$ where $0 \leq i,j\leq n-1$ as the set of generators with the following relations:

\begin{equation}\label{eq:z-and-tau}
 z'.z' = \wun' \qquad z'\wtau_{i,j}' = \wtau_{i,j}' z' \qquad \wtau_{i,j}' = z'\wtau_{j,i}' \qquad \wtau_{i,j}'.\wtau_{i,j}' = z' \qquad \text{ for any }i,j
\end{equation}

\begin{equation}\label{eq:anticommutation}
  \wtau_{i,j}'.\wtau_{k,l}'  = z' \wtau_{k,l}' . \wtau_{i,j}' \qquad \text{ for any } i,j,k,l \text{ if } \{i,j\}\cap \{k,l\}=\emptyset
\end{equation}

\begin{equation}\label{eq:conjugation}
  \wtau_{i,j}'.\wtau_{j,k}' . \wtau_{i,j}' = \wtau_{j,k}' . \wtau_{i,j}' . \wtau_{j,k}' = \wtau_{i,k}'\qquad  \text{ for any }i,j,k
\end{equation}

\begin{proof}
  Let $\wss_n$ the group with $z$ and $\wtau_i$ as generators and $\wss_{n}'$ the other one.
Define $\phi:\wss_n \rightarrow \wss_{n}'$ given on generators by $\phi(\wtau_i )=\wtau_{i,i+1}'$, $\phi(z)=z'$.
For $i<j$, let $\phi(\wtau_{i,j}) = \wtau_{i,j}'$.
By definition, \eqref{eq:z-and-tau} is verified.
Lemma \ref{lemma:conjugation} gives equations \eqref{eq:anticommutation} and \eqref{eq:conjugation}.
So the map $\phi$ extends to a group isomorphism.
\end{proof}

In what follows, we drop the prime exponent and only reffer to $\wtau_{i,j}$ and $z$ ($\wtau_i$ means $\wtau_{i,i+1}$).

%% file: definition-complex.tex
\section{The chain complex}\label{section:definition-complex}
Let $G$ be a grid presentation with complexity $n$ of the link $\overrightarrow L$.
Let $\Lambda$ denote the ring $\zzz[U_{O_1},\ldots,U_{O_n}]$.
We define $\widetilde{C}^-(G)$ to be the free $\Lambda$-module with generating set $\wss_n$ quotiented by the submodule generated by $\{z+1\}$ i.e.
$$
\widetilde{C}^-(G) = \Lambda[\wss_n] / <z+1>
$$
Considered as module, $\widetilde{C}^-(G)$ coincides with the free $\Lambda$-module with generating set $\ss_n$.
But we can also consider the structure of algebra of $\widetilde{C}^-(G)$ over $\Lambda$.
In this case, one can think of $\widetilde{C}^-(G)$ as the group algebra of $\ss_n$ over $\Lambda$ where the product is twisted by a non-trivial 2-cocycle (see section \ref{section:sign-assignment}).

We endow the set of generators with a Maslov grading $M$ and an Alexander filtration $A$ given by:
$$ M(\wx) = M( \x ) $$
$$ A(\wx) = A( \x ) $$

Let $\wx$ and $\wy$ be two elements of $\wss_n$.
The set of rectangles $\Rect(\wx,\wy)$ connecting $\wx$ to $\wy$ is the empty set if $\wy \neq  \wx.\wtau_{i,j}$ for all $i\neq j$, else it is $\wtau_{i,j}$.

If we consider the set $\Rect(\x,\y)$ of rectangles connecting $\x$ to $\y$ as in \cite{MOST}, either it is the empty set, or it consists of two rectangles.
We interpret the rectangle $\wtau_{i,j}$ in the oriented torus $\mathcal T$ as the rectangle whose bottom left corner belongs to the $i$-th vertical circle.
So in the case where $\Rect(\x,\y)=\{r_1,r_2\}$ the two corresponding rectangles are $\wtau_{i,j}$ and $\wtau_{j,i}$ and remember that $\wtau_{i,j} = z \wtau_{j,i}$.
Let $r$ be the rectangle of $\Rect(\x,\y)$ corresponding to $\wr$.
A rectangle $\wr\in \Rect(\widetilde{\x},\widetilde{\y})$ is said to be empty if $r \in \EmptyRect(\x,\y)$.
The set of empty rectangles connecting $\widetilde{\x}$ to $\widetilde{\y}$ is denoted $\EmptyRect(\widetilde{\x},\widetilde{\y})$.

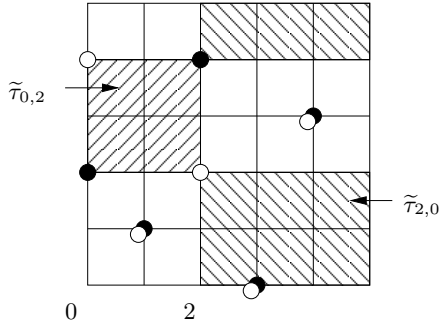
\begin{figure}[!ht]
  \begin{center}
	\input{deux-rectangles.pstex_t} 
  \end{center}
   \caption{\footnotesize \textbf{Rectangles.} Black dots represent $\x$ and white dots $\y$.
The two hatched regions correspond to rectangles $\wtau_{0,2} \in \Rect(\wx, \wx . \wtau_{0,2})$ and $\wtau_{2,0} \in \Rect(\wx ,  \wx . \wtau_{2,0})$.
The rectangle $\wtau_{0,2}$ is an empty rectangle while $\wtau_{2,0}$ is not.}
   \label{fig:deux-rectangles}
\end{figure}

We endow $\widetilde{C}^-(G)$ with a differential $\widetilde{\partial}^- $ given on elements of $\wss_n$ by:
$$
\widetilde{\partial}^- \wx = \sum_{\wy \in \wss_n}
\sum_{  \wr \in \EmptyRect (\widetilde{\x},\widetilde{\y})}
U^{O_{1}(\wr)}_{O_1} \ldots U^{O_{n}(\wr)}_{O_n} .\widetilde{\y}
$$
where $O_{k}(\wr)$ is the number of times $O_k$ appears in the interior of $r$.

\begin{proposition}
  The differential $\wdm $ drops the Maslov degree by one and respect the Alexander filtration.
\end{proposition}

\begin{proof}
  It is a straightforward consequence of calculus done in \cite{MOST}.
\end{proof}

\begin{proposition}
  The endomorphism $\wdm $ of $\widetilde{C}^- (G)$ is a differential, i.e.
$$\wdm \circ \wdm = 0$$
\end{proposition}

\begin{proof}
Let $\wx = s(\x) \in \wss_n$, viewed as a generator of $\widetilde{C}^-  (G)$.
Then
$$
\wdm\circ\wdm(\wx) =
\sum_{(\wy,\wz)\in \wss_n}
\sum_{\wr_2 \in \EmptyRect(\wy,\wz)}
\sum_{\wr_1 \in \EmptyRect(\wx,\wy)}
U_{O_1}^{O_1(\wr_1)+O_1(\wr_2)} \ldots U_{O_n}^{O_n(\wr_1)+O_n(\wr_2)}. \wz
$$

\begin{figure}[!ht]
  \begin{center}
	\input{differentielle.pstex_t} 
  \end{center}
   \caption{\footnotesize \textbf{$\wdm \circ \wdm = 0$.} }
   \label{fig:differentielle}
\end{figure}

There are different cases which are illustrated by figure \ref{fig:differentielle}.
\paragraph{\textbf{Cases 1,2,3.}}
The rectangles corresponding to $\wtau_{i,j}$ and $\wtau_{k,l}$ give the elements $\wz_1 =  \wx .\wtau_{k,l}.\wtau_{i,j} $ and $\wz_2 = \wx . \wtau_{i,j}.\wtau_{k,l}$.
By equation \eqref{eq:anticommutation} contribution to $\wdm \circ \wdm (\wx)$ is null.

\paragraph{\textbf{Case 4.}}
Supports of the rectangles have a common edge.
The two corresponding elements are $\wz_1 =  \wx .\wtau_{i,j}.\wtau_{j,k}$ and $\wz_2 =  \wx . \wtau_{i,k}.\wtau_{i,j}$ with $i<j<k$.
By equation \eqref{eq:conjugation}, $\wz_1 = z \wz_2$ and so the contribution is null.
Other cases work in a similar way.

\paragraph{\textbf{Case 5.}}
The vertical annulus is of width 1 and corresponds to $\wz_1 =  U_{O_m}.\wx . \wtau_{i}.\wtau_{i}$ (it is a consequence of the condition on rectangles to be empty).

To this vertical annulus corresponds the horizontal annulus of height 1 which contains $O_m$.
This horizontal annulus contributes for $U_{O_m}. \wx . \wtau_{l,k} . \wtau_{k,l} = U_{O_m}.\wx$ for a pair $k<l \in \{0,\ldots,n-1\}$.
So, the contribution of each vertical annulus is canceled by the corresponding horizontal annulus.
The global contribution to $\wdm \circ \wdm (\wx)$ is null.
\end{proof}

%% file: deux-rectangles.pstex_t
\begin{picture}(0,0)%
\includegraphics{deux-rectangles.pstex}%
\end{picture}%
\setlength{\unitlength}{3108sp}%
\begingroup\makeatletter\ifx\SetFigFont\undefined%
\gdef\SetFigFont#1#2#3#4#5{%
  \reset@font\fontsize{#1}{#2pt}%
  \fontfamily{#3}\fontseries{#4}\fontshape{#5}%
  \selectfont}%
\fi\endgroup%
\begin{picture}(3180,2596)(-644,-1970)
\put(766,-1906){\makebox(0,0)[lb]{\smash{{\SetFigFont{9}{10.8}{\familydefault}{\mddefault}{\updefault}{\color[rgb]{0,0,0}$2$}%
}}}}
\put(-179,-1906){\makebox(0,0)[lb]{\smash{{\SetFigFont{9}{10.8}{\familydefault}{\mddefault}{\updefault}{\color[rgb]{0,0,0}$0$}%
}}}}
\put(-629,-151){\makebox(0,0)[lb]{\smash{{\SetFigFont{9}{10.8}{\familydefault}{\mddefault}{\updefault}{\color[rgb]{0,0,0}$\wtau_{0,2}$}%
}}}}
\put(2521,-1051){\makebox(0,0)[lb]{\smash{{\SetFigFont{9}{10.8}{\familydefault}{\mddefault}{\updefault}{\color[rgb]{0,0,0}$\wtau_{2,0}$}%
}}}}
\end{picture}%

%% file: differentielle.pstex_t
\begin{picture}(0,0)%
\includegraphics{differentielle.pstex}%
\end{picture}%
\setlength{\unitlength}{3108sp}%
\begingroup\makeatletter\ifx\SetFigFont\undefined%
\gdef\SetFigFont#1#2#3#4#5{%
  \reset@font\fontsize{#1}{#2pt}%
  \fontfamily{#3}\fontseries{#4}\fontshape{#5}%
  \selectfont}%
\fi\endgroup%
\begin{picture}(8322,6116)(-959,-5341)
\put(5176,-4381){\makebox(0,0)[lb]{\smash{{\SetFigFont{9}{10.8}{\familydefault}{\mddefault}{\updefault}{\color[rgb]{0,0,0}$l$}%
}}}}
\put(4186,-4381){\makebox(0,0)[lb]{\smash{{\SetFigFont{9}{10.8}{\familydefault}{\mddefault}{\updefault}{\color[rgb]{0,0,0}$k$}%
}}}}
\put(4681,-5281){\makebox(0,0)[lb]{\smash{{\SetFigFont{9}{10.8}{\familydefault}{\mddefault}{\updefault}{\color[rgb]{0,0,0}Case 5}%
}}}}
\put(-944,-1186){\makebox(0,0)[lb]{\smash{{\SetFigFont{9}{10.8}{\familydefault}{\mddefault}{\updefault}{\color[rgb]{0,0,0}$i$}%
}}}}
\put(-494,-1186){\makebox(0,0)[lb]{\smash{{\SetFigFont{9}{10.8}{\familydefault}{\mddefault}{\updefault}{\color[rgb]{0,0,0}$j$}%
}}}}
\put(-44,-691){\makebox(0,0)[lb]{\smash{{\SetFigFont{9}{10.8}{\familydefault}{\mddefault}{\updefault}{\color[rgb]{0,0,0}$k$}%
}}}}
\put(1261,-646){\makebox(0,0)[lb]{\smash{{\SetFigFont{9}{10.8}{\familydefault}{\mddefault}{\updefault}{\color[rgb]{0,0,0}$l$}%
}}}}
\put(2386,-871){\makebox(0,0)[lb]{\smash{{\SetFigFont{9}{10.8}{\familydefault}{\mddefault}{\updefault}{\color[rgb]{0,0,0}$i$}%
}}}}
\put(2791,-1321){\makebox(0,0)[lb]{\smash{{\SetFigFont{9}{10.8}{\familydefault}{\mddefault}{\updefault}{\color[rgb]{0,0,0}$k$}%
}}}}
\put(4411,-871){\makebox(0,0)[lb]{\smash{{\SetFigFont{9}{10.8}{\familydefault}{\mddefault}{\updefault}{\color[rgb]{0,0,0}$j$}%
}}}}
\put(5761,-961){\makebox(0,0)[lb]{\smash{{\SetFigFont{9}{10.8}{\familydefault}{\mddefault}{\updefault}{\color[rgb]{0,0,0}$k$}%
}}}}
\put(7201,-961){\makebox(0,0)[lb]{\smash{{\SetFigFont{9}{10.8}{\familydefault}{\mddefault}{\updefault}{\color[rgb]{0,0,0}$l$}%
}}}}
\put(5536,-556){\makebox(0,0)[lb]{\smash{{\SetFigFont{9}{10.8}{\familydefault}{\mddefault}{\updefault}{\color[rgb]{0,0,0}$i$}%
}}}}
\put(4501,-4741){\makebox(0,0)[lb]{\smash{{\SetFigFont{9}{10.8}{\familydefault}{\mddefault}{\updefault}{\color[rgb]{0,0,0}$i$}%
}}}}
\put(4996,-4741){\makebox(0,0)[lb]{\smash{{\SetFigFont{9}{10.8}{\familydefault}{\mddefault}{\updefault}{\color[rgb]{0,0,0}$i+1$}%
}}}}
\put(3916,-1366){\makebox(0,0)[lb]{\smash{{\SetFigFont{9}{10.8}{\familydefault}{\mddefault}{\updefault}{\color[rgb]{0,0,0}$l$}%
}}}}
\put(6076,-511){\makebox(0,0)[lb]{\smash{{\SetFigFont{9}{10.8}{\familydefault}{\mddefault}{\updefault}{\color[rgb]{0,0,0}$j$}%
}}}}
\put(586,-4831){\makebox(0,0)[lb]{\smash{{\SetFigFont{9}{10.8}{\familydefault}{\mddefault}{\updefault}{\color[rgb]{0,0,0}$i$}%
}}}}
\put(1261,-4831){\makebox(0,0)[lb]{\smash{{\SetFigFont{9}{10.8}{\familydefault}{\mddefault}{\updefault}{\color[rgb]{0,0,0}$j$}%
}}}}
\put(2116,-4831){\makebox(0,0)[lb]{\smash{{\SetFigFont{9}{10.8}{\familydefault}{\mddefault}{\updefault}{\color[rgb]{0,0,0}$k$}%
}}}}
\put(181,-1591){\makebox(0,0)[lb]{\smash{{\SetFigFont{9}{10.8}{\familydefault}{\mddefault}{\updefault}{\color[rgb]{0,0,0}Case 1}%
}}}}
\put(3331,-1546){\makebox(0,0)[lb]{\smash{{\SetFigFont{9}{10.8}{\familydefault}{\mddefault}{\updefault}{\color[rgb]{0,0,0}Case 2}%
}}}}
\put(6166,-1501){\makebox(0,0)[lb]{\smash{{\SetFigFont{9}{10.8}{\familydefault}{\mddefault}{\updefault}{\color[rgb]{0,0,0}Case 3}%
}}}}
\put(1351,-5326){\makebox(0,0)[lb]{\smash{{\SetFigFont{9}{10.8}{\familydefault}{\mddefault}{\updefault}{\color[rgb]{0,0,0}Case 4}%
}}}}
\put(4771,-3886){\makebox(0,0)[lb]{\smash{{\SetFigFont{9}{10.8}{\familydefault}{\mddefault}{\updefault}{\color[rgb]{0,0,0}$O_m$}%
}}}}
\end{picture}%

%% file: properties.tex
\section{Properties of the chain complex}\label{section:properties}
\begin{proposition}\label{proposition:mod-two-reduction}
The tensor product of the filtered chain complex $\widetilde{C}^- (G)$ with $\zzz/2\zzz$ over $\zzz$ is isomorphic to the filtered chain complex $C^- (G)$ with coefficients in $\zzz/2 \zzz$.
In particular
$$H_* (\widetilde{C}^- (G)  \otimes \zzz/2\zzz ) \cong H_* (C^- (G) ; \zzz/2\zzz )$$
\end{proposition}

\begin{proof}
It is a consequence of the construction of the complex $\widetilde{C}^- (G)$.
\end{proof}

\begin{lemma}\label{lemma:homotopic-multiplication}
  Suppose that $O_i$ and $O_j$ belongs to the same component of $\overrightarrow L$.
Then multiplication by $U_{O_i}$ is filtered chain homotopic to multiplication by $U_{O_j}$.
\end{lemma}

\begin{proof}
The proof is the same as in Proposition 2.9 \cite{MOST}.
\end{proof}

\begin{theorem}\label{theo:filtered-quasi-iso-type}
  Let $\overrightarrow{L}$ be an oriented link with $l$ components.
Number the set $\Os=\{O_i\}_{i=1}^n$ such that $O_1,\ldots,O_l$ correspond to the different components of $\overrightarrow{L}$.
Then, the filtered quasi-isomorphism type of $(\widetilde{C}^- (G),\wdm)$ viewed over $\zzz[U_{O_1},\ldots,U_{O_l}]$ is an invariant of the link.
\end{theorem}

\begin{proof}
Multiplication by $U_{O_i}$ is filtered homotopic to multiplication by $U_{O_j}$ if $O_i$ and $O_j$ belong to the same component of the link (lemma \ref{lemma:homotopic-multiplication}).
It allows us to take for each homology class a representative in the variables $U_{O_1},\ldots,U_{O_l}$.

Different grid diagrams can lead to the same link.
However, given a link $\overrightarrow{L}$ and two grid diagrams $G$ and $H$ of $\overrightarrow{L}$, one can obtain $H$ starting with $G$ by a finite sequence of elementary moves which are cyclic permutation, commutation and (de)-stabilization (\cite{cromwell}, \cite{dynnikov}).
We now prove that the filtered quasi-isomorphism type is unchanged after making any elementary move.

\paragraph{\textbf{Cyclic permutation.}}
Let $\widetilde{\sigma} = \wtau_0 .\wtau_1 . \ldots . \wtau_{n-2}$ be an element in $\wss_n$.
So $\sigma$ is the cyclic pertumation $(01\ldots(n-1) )$.
There are two cases:
\begin{enumerate}
  \item \textbf{Vertical cyclic permutation.}
Suppose that the grid is moved one step upper.
Denote $G$ the grid diagram before the move and $H$ after the move.
Define
$$\Phi:\widetilde{C}^- (G) \rightarrow \widetilde{C}^- (H)$$
given on generators by
$$\Phi(\wx) = \widetilde{\sigma}.\wx$$
The application $\Phi$ induces an application on rectangles $\Phi_{\Rect}:\Rect_G \rightarrow \Rect_H$ which is the identity i.e. if $\wr \in \Rect(\wx,\wy)$ then $\wr \in \Rect(\widetilde{\sigma}.\wx,\widetilde{\sigma}.\wy)$.
Then $\Phi$ is a filtered isomorphism of chain complexes:
$$
\begin{array}{rcl}
  \Phi \circ \wdm \wx & = & \Phi( \sum_{\wy \in \wss_n}
\sum_{  \wr \in \EmptyRect (\widetilde{\x},\widetilde{\y})}
U^{O_{1}(\wr)}_{O_1} \ldots U^{O_{n}(\wr)}_{O_n} .\wy  ) \\

& = & \sum_{\wy \in \wss_n}
\sum_{  \wr \in \EmptyRect (\widetilde{\x},\widetilde{\y})}
U^{O_{1}(\wr)}_{O_1} \ldots U^{O_{n}(\wr)}_{O_n} .\widetilde{\sigma}.\wy \\

& = & \sum_{\wy \in \wss_n}
\sum_{  \wr' \in \EmptyRect (\widetilde{\sigma}.\widetilde{\x},\widetilde{\sigma}.\widetilde{\y})}
U^{O_{1}(\wr')}_{O_1} \ldots U^{O_{n}(\wr')}_{O_n} .\widetilde{\sigma}.\wy \\

& = & \sum_{\wy' \in \wss_n}
\sum_{  \wr' \in \EmptyRect (\widetilde{\sigma}.\widetilde{\x},\widetilde{\y}')}
U^{O_{1}(\wr')}_{O_1} \ldots U^{O_{n}(\wr')}_{O_n} .\wy' \\

& = & \wdm \circ \Phi(\wx) \\
\end{array}
$$

  \item \textbf{Horizontal cyclic permutation.}
Let $G$ the initial diagram and $H$ the one obtained after horizontal cyclic permutation (we consider the case where the grid is moved left, the other case is similar).
Define $\Phi:\widetilde{C}^- (G) \rightarrow \widetilde{C}^- (H)$ given on generators by
$$
\Phi(\wx) = (-1)^{\varepsilon(\sigma) . \varepsilon(\x)} .\wx . \widetilde{\sigma}^{-1}
$$
where the signature maps to the group $\{0,1\}$.
\begin{lemma}
  $\Phi$ is a filtered isomorphism of chain complexes.
\end{lemma}

\begin{proof}
The map $\Phi$ induces on rectangles the map $\Phi_{\Rect}:\Rect_G \rightarrow \Rect_H$ given by
$$
\Phi_{\Rect}(\wtau_{i,j}) = (-1)^{\varepsilon(\sigma)} . \widetilde{\sigma}.\wtau_{\sigma(i),\sigma(j)} .\widetilde{\sigma}^{-1}
$$
(see lemma \ref{lemma:conjugation}).
Consider a summand in $\wdm \circ \Phi(\wx)$ (we forget the corresponding variables $U_k$):
$$
\begin{array}{rcl}
  (-1)^{\varepsilon(\sigma) . \varepsilon(\x)} . \wx.\widetilde{\sigma}^{-1}.\wtau_{i,j} & = & 
(-1)^{\varepsilon(\sigma) . \varepsilon(\x)} . \wx . \widetilde{\sigma}^{-1} . (-1)^{\varepsilon(\sigma)} . \widetilde{\sigma}.\wtau_{\sigma(i),\sigma(j)} .\widetilde{\sigma}^{-1} \\

& = & (-1)^{\varepsilon(\sigma) . (\varepsilon(\x)+1)} . \wx . \wtau_{\sigma(i),\sigma(j)} .\widetilde{\sigma}^{-1} \\

\end{array}
$$
The corresponding summand in $\Phi\circ \wdm \wx$ is
$$
(-1)^{ \varepsilon(\sigma) . \varepsilon(\x.\tau_{\sigma(i),\sigma(j)}) }  . \wx . \wtau_{\sigma(i),\sigma(j) } .\widetilde{\sigma}^{-1}
$$
Since $\varepsilon(\x.\tau_{\sigma(i),\sigma(j)}) = \varepsilon(\x)+1$, we obtain
$$ (-1)^{ \varepsilon(\sigma) . \varepsilon(\x.\tau_{\sigma(i),\sigma(j)}) }  . \wx . \wtau_{\sigma(i),\sigma(j) } .\widetilde{\sigma}^{-1} = (-1)^{\varepsilon(\sigma) . \varepsilon(\x)} . \wx.\widetilde{\sigma}^{-1}.\wtau_{i,j}$$
\end{proof}
\end{enumerate}

\paragraph{\textbf{Commutation.}}
There are two cases which are commutation of columns and commutation of rows.
We deal in detail the case of commutation of columns.

\begin{itemize}
  \item \textbf{Commutation of columns.}
Let $G$ be a grid presentation of $\overrightarrow{L}$ and $H$ the grid diagram obtained from $G$ after commutation.
Let $\beta$ be the vertical circle of $G$ and $\alpha$ be the one for $H$.
We represent the two diagrams $G$ and $H$ on the same torus (see figure \ref{fig:pentagone-gauche-droit}).

Now, we define a filtered chain morphism
$$\Phi_{\beta \gamma}:\widetilde{C}^- (G) \rightarrow \widetilde{C}^- (H)$$
by counting pentagons.

We recall the definition in \cite{MOST}.
Let $\x\in C^-(G)$ and $\y\in C^-(H)$ be two generators.
We denote $\Pent_{\beta \gamma}(\x,\y)$ the set of immersed pentagons connecting $\x$ to $\y$.
It is the empty set if $\x$ and $\y$ differ by more than two points.
An element $p\in\Pent_{\beta \gamma}(\x,\y)$ is an immersed disk in $\mathcal T$, whose boundary consists of 5 arcs, each of them is contained in a vertical or a horizontal circle.
Moreover, with the orientation induced by $\mathcal T$ of the boundary of $p$, we begin by the point in $\x$ on the $\beta$ circle, we go through a horizontal arc to a point in $\y$, we follow a vertical arc to the corresponding point of $\x$, we go through a horizontal arc to a point in $\y$ and we go back to the beginning by going first on $\gamma$ and then on $\beta$ through one intersection point between $\alpha$ and $\beta$.
Moreover, we require that each angle of the pentagon be acute (see figure \ref{fig:pentagone-gauche-droit}).
Let $\Pent^{\circ}_{\beta \gamma}(\x,\y)$ be the set of empty pentagons i.e. if $p\in \Pent^{\circ}_{\beta \gamma}(\x,\y)$ then $\x \cap Int(p)=\emptyset$.

Let $\wx \in \widetilde{C}^- (G)$ and $\wy \in \widetilde{C}^- (H)$ be two generators.
We consider $\wy$ and $\wy$ as elements of the same group $\wss_n$ where we identify the $\beta$ circle with the $\gamma$ circle.
The set of pentagons $\Pent_{\beta \gamma}(\wx,\wy)$ connecting $\wx$ to $\wy$ is the empty set if $\Pent_{\beta \gamma}(\x,\y)=\emptyset$.
Otherwise $\Pent_{\beta \gamma}(\wx,\wy)=\wtau_{i,j}$ where the bottom left corner of the pentagon is in the $i$-th circle and if $\wy= \wx .\wtau_{i,j}$.
A pentagon $\widetilde{p} \in \Pent_{\beta \gamma}(\wx,\wy)$ is said empty if $p \in \Pent^{\circ}_{\beta \gamma}(\x,\y)$.
Let $\Pent^{\circ}_{\beta \gamma}(\wx,\wy)$ be the set of empty pentagons connecting $\wx$ to $\wy$.

Define
$$
\varepsilon_{\beta \gamma}:\Pent^{\circ}_{\beta \gamma}(\wx,\wy)\rightarrow \{\pm 1 \}
$$
by $\varepsilon_{\beta \gamma}(\widetilde{p})=+1$ if $p$ is a left pentagon and $-1$ if it is a right pentagon.
$$
\Phi_{\beta \gamma}(\wx)=\sum_{\wy \in \wss_n} \sum_{\widetilde{p}\in \Pent^{\circ}_{\beta \gamma}(\wx,\wy)} (-1)^{ M(\wx) }\varepsilon_{\beta \gamma}(\widetilde{p}).U_{O_1}^{O_1 (\widetilde{p})}\dots U_{O_n}^{O_n (\widetilde{p})} . \wy
$$
where $O_m (\widetilde{p})$ is the number of times $O_m$ appears in the interior of $p$.

Define $\varepsilon_{ \gamma \beta}$ with the same convention: left pentagons takes value $+1$ and right pentagons takes value $-1$.

\begin{rembis}
  Instead of taking $(-1)^{ M(\wx) }$, one can take $(-1)^{\varepsilon(\x)}$ as in the proof of cyclic permutation.
In the following, we take the Maslov degree but we can think of it as the signature.
\end{rembis}

\begin{figure}[!ht]
  \begin{center}
	\input{pentagone-gauche-droit.pstex_t} 
  \end{center}
   \caption{\footnotesize \textbf{Pentagons.} We mark with black dots $\x$ and with white dots $\y$.
The pentagon $\wtau_{i,j} \in\Pent^{\circ}_{\beta \gamma}(\wx,\wx.\wtau_{i,j}) $ is a left pentagon and so $\varepsilon_{\beta \gamma}(\wtau_{i,j} )=+1$.
The pentagon $\wtau_{i,k}) \in \Pent^{\circ}_{\beta \gamma}(\wx,\wx. \wtau_{i,k} )$ is a right pentagon and so $\varepsilon_{\beta \gamma}( \wtau_{i,k} )=-1$.}
   \label{fig:pentagone-gauche-droit}
\end{figure}
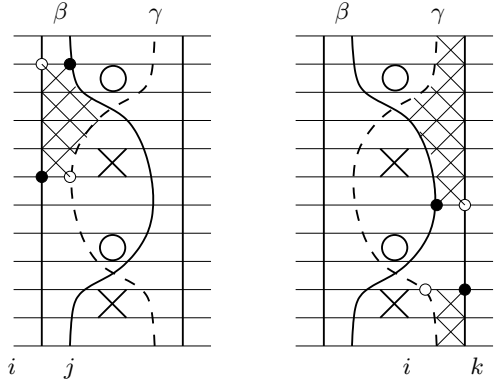

\begin{lemma}
  $\Phi_{\beta \gamma}$ is a filtered chain morphism.
\end{lemma}

\begin{proof}
Summands appearing in $\wdm \circ \Phi_{\beta \gamma}(\wx)$ are of this type: $\varepsilon_{\beta \gamma}(\widetilde{p})  (-1)^{ M(\wx) } . \wx . \widetilde{p} . \wr $, and those in $\Phi_{\beta \gamma}\circ \wdm \wx$ are of this type: $\varepsilon_{\beta \gamma}(\widetilde{p}')  (-1)^{ M(\wx.\wr') } . \wx  . \wr' . \widetilde{p}'$.
Each term in $\wdm \circ \Phi_{\beta \gamma}(\wx)$ corresponds to a term in $\Phi_{\beta \gamma}( \wdm \wx)$.
In case of figure \ref{fig:phi-morphisme-1}, we have in fact $\wr' . \widetilde{p}' = \widetilde{p} . \wr $.
Nevertheless, there are a left pentagon (say $\widetilde{p}'$) and a right pentagon (say $\widetilde{p})$ so via $\varepsilon_{\beta \gamma}$ they take different values which are compensated by the Maslov degree $M(\wx) = M(\wx.\wr) +1$.

In other cases except the special case illustrated by figure \ref{fig:phi-morphisme-2}, the two pentagons are on the same side but $\widetilde{p} . \wr = z \wr' . \widetilde{p}'$ and it works similary.

There is a special case illustrated by figure \ref{fig:phi-morphisme-2} and which appears once in $\wdm \circ \Phi_{\beta \gamma}(\wx)$ and once in $\Phi_{\beta \gamma} \circ \wdm \wx$.
With the same discussion as above we check that signs behave well.

\begin{figure}[!ht]
  \begin{center}
	\input{phi-morphisme-1.pstex_t} 
  \end{center}
   \caption{ \footnotesize \textbf{$\Phi_{\beta \gamma}$ filtered chain morphism.} The hatched domain is decomposed either as a left pentagon followed by a rectangle, or a rectangle followed by a right pentagon.
The first decomposition appears in $\wdm \circ \Phi_{\beta \gamma}$ and the second one appears in $\Phi_{\beta \gamma} \circ \wdm$. }
   \label{fig:phi-morphisme-1}
\end{figure}
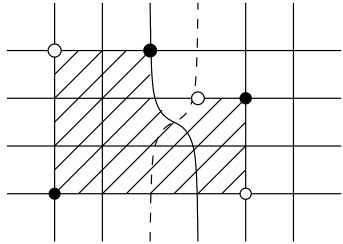

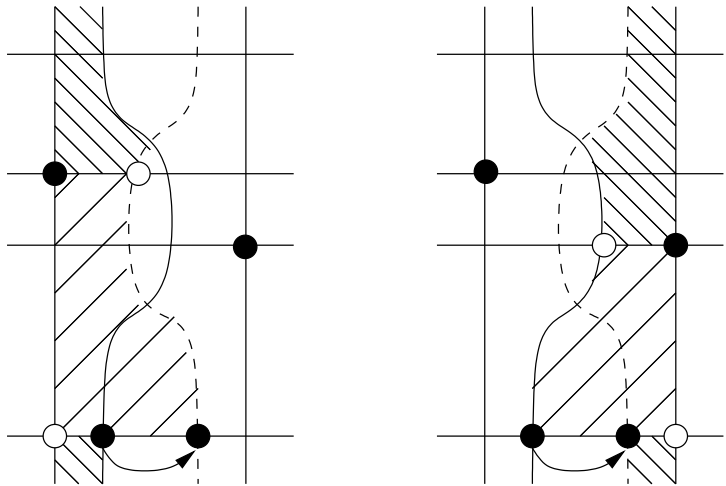
\begin{figure}[!ht]
  \begin{center}
	\input{phi-morphisme-2.pstex_t} 
  \end{center}
   \caption{ \footnotesize \textbf{$\Phi_{\beta \gamma}$ filtered chain morphism: special case.} Generators marked by black dots differ only in one point.
On the left, we have a left pentagon followed by a rectangle, on the right we have a rectangle followed by a right pentagon.}
   \label{fig:phi-morphisme-2}
\end{figure}

\end{proof}

To prove that $\Phi_{\beta \gamma}$ is a filtered homotopy equivalence we define an application which raises the maslov degree by one and counts empty hexagons.

First, we recall the definition of hexagons in the sense of \cite{MOST}.
Let $\x,\y \in C^- (G)$ be two generators.
Denote $\Hex_{\beta \gamma \beta}(\x,\y)$ the set of immersed hexagons connecting $\x$ to $\y$.
It is the empty set if $\y.\x^{-1}$ is not a transposition.
An element $h \in \Hex_{\beta \gamma \beta}(\x,\y)$ is an immersed disk in $\mathcal T$ with the orientation induced by $\mathcal T$.

Following the oriented boundary of $h$, we begin with the point of $\x$ which is on $\beta$, we go through a horizontal arc to a point of $\y$, through a vertical arc to a point of $\x$, through a horizontal arc to the point of $\y$ in $\beta$ and finally go back to the initial point of $\x$ through the vertical arc $\beta$ and $\gamma$ and the two intersection points between $\beta$ and $\gamma$ (see figure \ref{fig:h-morphisme}).
We require that all angles of the hexagon be acute.
The set $\Hex_{\beta \gamma \beta}^{\circ}$ of empty pentagons is those for which the interior doesn't contain any point of the corresponding generators.

The set of hexagons $\Hex_{\beta \gamma \beta}(\wx,\wy)$ connecting $\wx$ to $\wy$ is the empty set if $\Hex_{\beta \gamma \beta}(\x,\y)= \emptyset$.
Otherwise $\Hex_{\beta \gamma}(\wx,\wy)=\wtau_{i,j}$ where the bottom left corner of the pentagon is in the $i$-th circle and if $\wy= \wx .\wtau_{i,j}$.
A hexagon $\widetilde{h} \in \Hex_{\beta \gamma \beta}(\wx,\wy)$ is said empty if $h \in \Hex^{\circ}_{\beta \gamma \beta}(\x,\y)$.
Let $\Hex^{\circ}_{\beta \gamma \beta}(\wx,\wy)$ be the set of empty hexagons connecting $\wx$ to $\wy$.

The morphism $H_{\beta \gamma \beta}:\widetilde{C}^- (G) \rightarrow \widetilde{C}^- (G)$ is given on generators by
$$
H_{\beta \gamma \beta}(\wx)=\sum_{\wy \in \wss_n} \sum_{\widetilde{h} \in \Hex^{\circ}_{\beta \gamma \beta}(\wx,\wy)} U_{O_1}^{O_1 (\widetilde{h})}\dots U_{O_n}^{O_n (\widetilde{h})} . \wy
$$
where $O_m (\widetilde{h})$ is the number of times $O_m$ appears in the interior of $h$.
In particular $H_{\beta \gamma \beta}$ raise the Maslov degree by 1.

\begin{lemma}
  The application $\Phi_{\beta \gamma}$ induces a filtered isomorphism in homology; more precisely
$$
\Phi_{\gamma \beta} \circ \Phi_{\beta \gamma}  + id = \wdm \circ H_{\beta \gamma \beta } + H_{\beta \gamma \beta } \circ \wdm
$$
$$
\Phi_{\beta \gamma} \circ \Phi_{\gamma \beta }  + id = \wdm \circ H_{\gamma \beta \gamma} + H_{\gamma \beta \gamma} \circ \wdm
$$
\end{lemma}

\begin{proof}
By juxtaposing two pentagons which appears in $\Phi_{\gamma \beta} \circ \Phi_{\beta \gamma}$, we get a domain which has a unique alternative decomposition as a rectangle followed by a hexagon or a hexagon followed by a rectangle, counted in $\wdm \circ H_{\beta \gamma \beta}$ or in $ H_{\beta \gamma \beta} \circ \wdm$ (it is the case where the width of the rectangle is $> 1$).

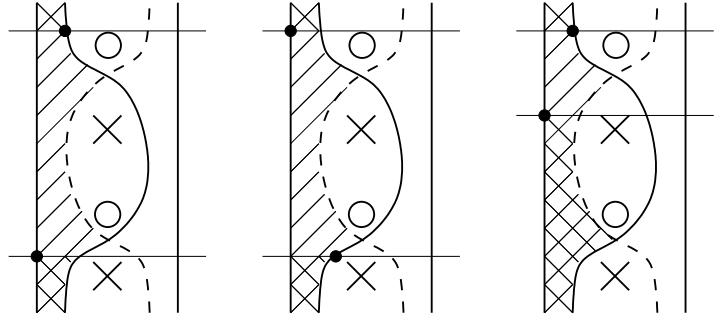
\begin{figure}[!ht]
  \begin{center}
	\input{h-morphisme.pstex_t}
  \end{center}
   \caption{ \footnotesize \textbf{Decomposing the identity map.} Consider the 3 configurations in $\widetilde{C}^- (G)$ marked by dark dots.
On the left, we have a hexagon followed by a rectangle: it appears in $\wdm \circ H_{\beta \gamma \beta}$.
In the middle, we have a rectangle followed by a hexagon: it appears in $ H_{\beta \gamma \beta} \circ \wdm$.
On the right, we have a pentagon followed by a pentagon: it appears in $\Phi_{\gamma \beta} \circ \Phi_{\beta \gamma}$.}
   \label{fig:h-morphisme}
\end{figure}

However, there is a domain which has a unique decomposition (see figure \ref{fig:h-morphisme}): it corresponds to a vertical annulus without any $X$.
It counts for $-id$ and admits a unique alternative decomposition.

When it is a pentagon followed by a pentagon, we start with $\wx$ and arrive to
$$
\varepsilon_{\beta \gamma}(\wp) .\varepsilon_{\gamma \beta}(\wp') . (-1)^{ M_G(\wx) + M_H(\wx.\wtau_i) } \wx. \wtau_i . \wtau_i = z \wx
$$
where $M_G$ (resp. $M_H$) is the Maslov degree on $G$ (resp. on $H$).
Since the two pentagons are on the same side, we have $\varepsilon_{\beta \gamma}(\wp) .\varepsilon_{\gamma \beta}(\wp') = +1$.
Moreover we have $M_G(\wx) = M_H( \wx.\wtau_i )$.

While starting with a hexagon and composing with a rectangle or starting with a rectangle and composing with a hexagon, we start with $\wx$ and arrive to $\wx .\wtau_i . \wtau_i = z\wx$.

In every case, starting with $\wx$ we get $z\wx$.
\end{proof}

\item \textbf{Commutation of rows.}
We deal the commutation of rows exactly the same way except that we must change the notions of \emph{left} and \emph{right} by \emph{top} and \emph{bottom}.
\end{itemize}

\paragraph{\textbf{Stabilisation.}} We follow the proof of invariance under stabilization done in \cite{MOST}, updating it to our context.
Let $G$ be the initial diagram of complexity $n$ and $H$ the one of complexity $n+1$ obtained after one stabilization move along the row containing $O_2$.
The set of $\Os$ in $G$ is numbered from $2$ to $n+1$.
We can suppose (after some cyclic permutations) that the new row and column are the $n+1$-th row and the $n+1$-th column.
We deal with the case drawn in figure \ref{fig:stabilisation-1}.

\begin{figure}[!ht]
  \begin{center}
	\input{stabilisation-1.pstex_t}
  \end{center}
   \caption{ \footnotesize \textbf{Stabilisation.} The point $x_0$ is the point belonging to all elements of $\mathbf{I}$ when we consider the set of generators of $\widetilde{C}^- (G)$ as a subset of the set of generators of $\widetilde{C}^- (H)$.
The $O$ and the $X$ are numbered $O_1$ and $X_1$.}
   \label{fig:stabilisation-1}
\end{figure}
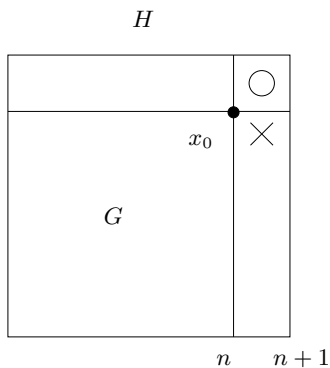

Let $B =\widetilde{C}^- (G)$ and $C = \widetilde{C}^- (H)$ the corresponding complexes.
First, we recall that $C'$ is the mapping cone of
$$
U_{O_1} - U_{O_2} : B[U_{O_1}] \rightarrow B[U_{O_1}]
$$
and the differential is given by $\partial' (a,b) = (\wdm a , (U_{O_1} - U_{O_2}). a-\wdm b )$.
Let $\mathcal L$ et $\mathcal R$ be the submodules of $C$ such that $C' = B[U_{O_1}] \bigoplus B[U_{O_1}] = \mathcal L \bigoplus \mathcal R$.
The module $\mathcal R$ heritates a Maslov degree and an Alexander filtration which is the one of $B[U_{O_1}]$ whereas $\mathcal L$ heritates the ones of $B[U_{O_1}]$ dropped by 1.

\begin{lemma}
  $C'$ is filtered quasi-isomorphic to $B$.
\end{lemma}

\begin{proof}
The proof is the same as in lemma 3.3 \cite{MOST}.
\end{proof}

It remains to prove that $C'$ is filtered quasi-isomorphic to $C$.
We define an injection which maps each generators of $B$ to a subset $\mathbf I$ of the set of generators of $C$.
It is $\phi:\wss_n \hookrightarrow \wss_{n+1}$ which maps $\wtau_i$ to $\wtau_i$ for all $i\in\{0,\ldots,n-2\}$.
We extend it on all elements of $\wss_n$ by making it compatible with the normal form previously described in section \ref{section:algebraic-preliminaries}.
In particular, if $\wy\in \mathbf{I}$ then $\x_0$ is a point of $\y$.
We have the following equalities for $\wx \in B$:
$$
M_B(\wx) = M_C(\phi(\wx))+1 = M_{C'}(0,\phi(\wx) ) = M_{C'}(\phi(\wx),0)+1
$$
$$
A_B(\wx) = A_C(\phi(\wx))+A(U_1) = A_{C'}(0,\phi(\wx) ) = A_{C'}(\phi(\wx),0)+A(U_1)
$$

We recall the definition of domains of type $L$ and $R$ (see \cite{MOST}).
Let $\x,\y$ be two generators in $\ss_n$.
A path from $\x$ to $\y$ is an oriented closed path $\gamma$ made of arcs in the circles $\alphas$ and $\betas$ which angles are points in $\x\cup \y$, oriented so that $\partial (\gamma \cap \alpha ) =\x-\y$.
A domain $p$ connecting $\x$ to $\y$ is a two chain in $\mathcal T$ whose boundary $\partial p$ is a path from $\x$ to $\y$ and denote $\pi(\x,\y)$ the set of domains connecting $\x$ to $\y$.

\begin{definition}
Let $\x\in \ss_n$ and $\y\in \ss_n \subset \ss_{n+1}$ (the inclusion is given by $\phi$).
A domain $p\in \pi(\x,\y)$ is said to have type $L$ (resp. type $R$) if it is trivial (in this case it has type $L$) or if it satisfies the following conditions:
\begin{itemize}
  \item $p$ has no negative local multiplicity,
\item for each $c\in \x\cup\y$ other than $x_0$, at least three of the four adjoining squares have local multiplicity equal to zero,
\item in a neighboorhood of $x_0$ the local multiplicity in three of the adjoining squares are $k$.
When $p$ has type $L$, the lower left corner has local multiplicity $k-1$, while for $p$ of type $R$, the lower right corner has multiplicity $k+1$.
\item $\partial p$ is connected.
\end{itemize}
The complexity of the trivial domain is 1 and for any other domains it is the number of horizontal arcs of its boundary.
The set of domains of type $L$ (resp. $R$) is denoted $\pi^{L}(\x,\y)$ (resp. $\pi^{R}(\x,\y)$).
We denote $\pi^{F}(\x,\y)= \pi^{L}(\x,\y) \cup \pi^{R}(\x,\y)$ and called its elements of type $F$.
\end{definition}

We wish to define a map $F:C \rightarrow C'$ which maps generators of $C$ to generators of $B$ through domains of type $R$ or $L$.
To do this, we define type $F$ domains $\pi^{F}(\wx,\wy)$ for $\wx \in C$ and $\wy \in \mathbf{I}$.

To a domain $p\in \pi^{F}(\x,\y)$ of complexity $m$ we associate a standard decomposition in a finite sequence of rectangles.
Since $\partial p$ is a connected oriented curve, we number the $\betas$-circles by $\{v_i\}_{i=1}^m$ in such a way that they inherits the cyclic order given by $\partial p$, and such that $v_m$ does not contain $x_0$.
We decompose $p$ in a sequence of rectangles $\{r_i\}_{i=1}^{m-1}$ so that $r_i$ is the rectangle between the vertical circles $v_i$ and $v_m$ (see figure \ref{fig:decomposition-polygone}).

We define $\pi^{F}(\wx,\wy)$ as the empty set if $\pi^{F}(\x,\y)=\emptyset$ , otherwise $\wp \in \pi^{F}(\wx,\wy)$ is
$$
\wr_1 . \ldots .\wr_{m-1}
$$
corresponding to the rectangles $\{r_i\}_{i=1}^{m-1}$ if $\y = \x . r_1 . \ldots .r_{m-1}$.

Define on the set of generators of $C$ the two applications $F^L : C\rightarrow \mathcal L$ and $F^R : C\rightarrow \mathcal R$ by
$$
F^L (\wx) = \sum_{ \wy\in \mathbf{I} } \sum_{ \wp \in \pi^{L}(\wx,\wy) } U_{O_2}^{O_2 (\wp)}\ldots U_{O_n}^{O_n (\wp)} . \wy
$$

$$
F^R (\wx) = \sum_{ \wy\in \mathbf{I} } \sum_{ \wp \in \pi^{R}(\wx,\wy) } U_{O_2}^{O_2 (\wp)}\ldots U_{O_n}^{O_n (\wp)} . \wy
$$
Putting those two applications together we get
$$
F = \left( \begin{array}{c}
  F^L \\
F^R\\
\end{array}\right)
:C\rightarrow C'
$$

Before proving that $F$ is a filtered chain morphism, we need to deal with the differential in $C'$.
The differential in $C'$ counts rectangles which are either empty (type 1) or empty except of the point $x_0$ (type 2).
For type 2 rectangles the variable $U_1$ doesn't appear in the differential.
Let $\wr$ be a rectangle in $B$.
Viewed as a rectangle in $C'$, if it is of type 1 we note $\wr'$ the corresponding rectangle (under the injection $\phi$) and its standard decomposition is $D_0(\wr') =\wr'$.
If it is of type 2, it corresponds to $\wr'$ and admits a unique decomposition shown in figure \ref{fig:decomposition-polygone} in three rectangles. $D_0 (\wr') = \wr_1 ' . \wr_2 ' . \wr_3 '$ is its standard decomposition.

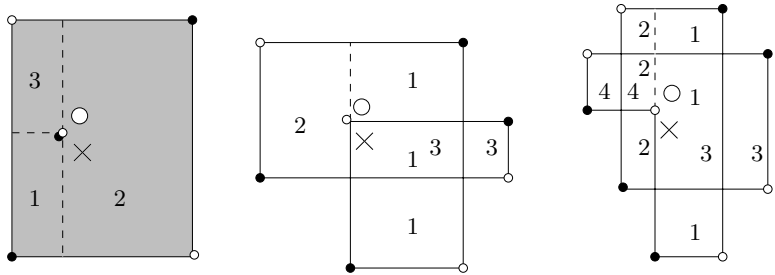
\begin{figure}[!ht]
  \begin{center}
	\input{decomposition-polygone.pstex_t}
  \end{center}
   \caption{ \footnotesize \textbf{Decomposing polygons.} On the left, we have the standard decomposition of a rectangle of type 2 as $\wr_1  . \wr_2  . \wr_3 $.
In the middle, we have the standard decomposition of a polygonal domain of type $R$ of complexity 4 as $\wr_1  . \wr_2  . \wr_3 $.
On the right, we have the standard decomposition of a polygonal domain of type $L$ of complexity 5 as $\wr_1  . \wr_2  . \wr_3 . \wr_4$.
Some regions are in the supports of more than one rectangle and are labeled with the number of each corresponding rectangles. }
   \label{fig:decomposition-polygone}
\end{figure}

\begin{lemma}
  Let $\wx,\wy$ be in $B$ and $\wx',\wy' \in \mathbf{I}$ their image in $C$.
Suppose there exists $\wr'\in \Rect(\wx',\wy')$ a rectangle of type 1 or 2.
Then the corresponding rectangle $\wr \in \Rect^{\circ}(\wx,\wy)$ satisfies:
\begin{enumerate}
  \item $\wr =\wr'$ if $\wr'$ has type 1.
  \item $\wr = \wr_1 ' . \wr_2 ' . \wr_3 '$ if $\wr'$ has type 2 and $D_0 (\wr') = \wr_1 ' . \wr_2 ' . \wr_3 '$.
\end{enumerate}
\end{lemma}

\begin{proof}
For rectangles of type 1, there is nothing to prove since $\x_0$ has coordonates $(n,n)$.

For rectangles of type 2, we must prove that for all $0\leq i,j < n-1$
$$
\wtau_{i,j} = \wtau_{i,n} . \wtau_{n,j} . \wtau_{i,n}
$$
but this is equation \eqref{eq:conjugation}.
\end{proof}

\begin{lemma}\label{lemma:F-morphism}
The map $F:C\rightarrow C'$ is a filtered chain morphism which respects the Maslov degree and preserves the Alexander filtration.
\end{lemma}

\begin{proof}
The proof is by cutting in two different ways domains appearing in $F\circ \wdm_C$ and $\wdm_{C'} \circ F$.

\begin{definition}
Let $\wp \in \pi^{F}(\wx,\wy)$ be a domain decomposed as $\wp = \wr_1 . \ldots .\wr_{n}$ with $\wr_i \in \Rect^{\circ}$ for all $i$.
If for some $i$ we have $r_i * r_{i+1} = r'_i *r'_{i+1}$ with $r'_i ,r'_{i+1}\in \Rect^{\circ}$ distinct from $r_i$ and $r_{i+1}$ then we say that the two decompositions of the domain $\wr_1 . \ldots  .\wr_i . \wr_{i+1} . \ldots .\wr_{n}$ and $\wr_1 . \ldots  .\wr'_i . \wr'_{i+1}.\ldots  .\wr_{n}$ differ by an elementary move (here $*$ is the juxtaposition of rectangle).
In particular $\wr_1 . \ldots  .\wr_i . \wr_{i+1} . \ldots .\wr_{n} = z \wr_1 . \ldots  .\wr'_i . \wr'_{i+1}.\ldots  .\wr_{n}$
\end{definition}

With this definition, the proof of the lemma \ref{lemma:F-morphism} is a rewriting of the proof of lemma 3.5 \cite{MOST}.
\end{proof}

\begin{proposition}\label{proposition:F-quasi-iso}
  $F$ is a filtered quasi-isomorphism.
\end{proposition}

\begin{proof}
The proof is in four steps which we recall here:
\begin{enumerate}
  \item Define a filtration $\mathcal F$ on the graded object associated to the Alexander filtration $A$ on complexes $\Caa$ and $\Caa'$ obtained from $C$ and $C'$.
  \item Describing the homology of the graded object associated to the filtration $\mathcal F$ and $A$.
  \item Proving that $\widetilde{F}_{gr\vert A,\mathcal F }$ is an isomorphism on the homology.
  \item Using a principle of homological algebra to conclude that $F$ is a filtered quasi-isomorphism.
\end{enumerate}

\paragraph{\textbf{Step 1.}} First, remark that the graded object associated to the Alexander filtration is just given by considering rectangles empty of $X\in \Xs$.
Let $\Caa$ be equal to $C_{\vert U_i=0 \forall i}$ (in particular, the corresponding differential counts only rectangles which do no contain any $\Os$).
Let $\Caa_{gr \vert A}$ be the associated graded object and $\Caa_{gr\vert A=h}$ the summand generated by generators $\wx\in \wss_n$ where $A(\wx)=h\in \zzz^l$.

Let $\mathcal Q$ be the $n\times n$ square corresponding to $G\subset H$ (see figure \ref{fig:stabilisation-1}): in particular, the row and column containing $O_1$ is not in $\mathcal Q$.
For any $\wx,\wy \in \wss_{n+1}$ and any domain $\wp\in \pi(\wx,\wy)$ such that $O_i(\wp)=X_i(\wp)=0$ for all $i$ then
$$
\mathcal{ F} (\wx) - \mathcal{ F} (\wy) = \# (\mathcal{ Q} \cap p )
$$
is well defined (see p. 23 \cite{MOST}) and determines a filtration on $\Caa_{gr\vert A=h}$.
Let $C_{gr \vert A,\mathcal{F}}$ be the graded object associated.

\paragraph{\textbf{Step 2.}} Let $\wss_{n+1} = \mathbf{I} \cup \mathbf{NI} \cup \mathbf{NN}$, where $\mathbf{NI}$ is the subset of $\wss_{n+1}$ whose elements $\wx\in \mathbf{NI}$ are such that $\x (n+1)= n$ and $\mathbf{NN}$ is the complement in $\wss_{n+1}$ of $\mathbf{I} \cup \mathbf{NI}$.

\begin{lemma}[lemma 3.7 \cite{MOST}]
  $H_* (\Caa_{gr \vert A,\mathcal{F}})$ is isomorphic to the free $\zzz$-module generated by elements of $\mathbf{I}$ and $\mathbf{NI}$.
\end{lemma}

\paragraph{\textbf{Step 3.}} Let $\widetilde{F}_{gr\vert A,\mathcal F }$ be the map induced by $F$:
$$
\widetilde{F}_{gr\vert A,\mathcal F } : \Caa_{gr \vert A,\mathcal{F}} \rightarrow \Caa'_{gr \vert A,\mathcal{F}}
$$
where $\Caa'_{gr \vert A,\mathcal{F}}$ splits as the direct sum $\widetilde{\mathcal{L}}_{gr \vert A,\mathcal{F}} \oplus \widetilde{\mathcal{R}}_{gr \vert A,\mathcal{F}}$ (the map $\times( U_{O_1}-U_{O_2})$ is equal to $0$ in $\Caa$), both of which are freely generated by elements in $\mathbf{I}$.
Then $\widetilde{F}_{gr\vert A,\mathcal Q }$ is a quasi-isomorphism (Proposition 3.8 \cite{MOST}).

\paragraph{\textbf{Step 4.}} Using the fact that a filtered chain map which induces an isomorphism on the homology of the associated graded objects is a filtered quasi-isomorphism gives that $\widetilde{F}$ is a quasi-isomorphism between $\Caa$ and $\Caa'$.
Using this fact one more time gives that $F$ is a quasi-isomorphism between $C$ and $C'$: $\Caa$ is just the graded object associated to $C$ with the filtration counting the number of variables $U_i$'s. 
\end{proof}

Thanks to proposition \ref{proposition:F-quasi-iso}, invariance under stabilization is proved.
\end{proof}

%% file: pentagone-gauche-droit.pstex_t
\begin{picture}(0,0)%
\includegraphics{pentagone-gauche-droit.pstex}%
\end{picture}%
\setlength{\unitlength}{3108sp}%
\begingroup\makeatletter\ifx\SetFigFont\undefined%
\gdef\SetFigFont#1#2#3#4#5{%
  \reset@font\fontsize{#1}{#2pt}%
  \fontfamily{#3}\fontseries{#4}\fontshape{#5}%
  \selectfont}%
\fi\endgroup%
\begin{picture}(3897,3079)(391,-2384)
\put(766,524){\makebox(0,0)[lb]{\smash{{\SetFigFont{9}{10.8}{\familydefault}{\mddefault}{\updefault}{\color[rgb]{0,0,0}$\beta$}%
}}}}
\put(1531,524){\makebox(0,0)[lb]{\smash{{\SetFigFont{9}{10.8}{\familydefault}{\mddefault}{\updefault}{\color[rgb]{0,0,0}$\gamma$}%
}}}}
\put(3016,524){\makebox(0,0)[lb]{\smash{{\SetFigFont{9}{10.8}{\familydefault}{\mddefault}{\updefault}{\color[rgb]{0,0,0}$\beta$}%
}}}}
\put(3781,524){\makebox(0,0)[lb]{\smash{{\SetFigFont{9}{10.8}{\familydefault}{\mddefault}{\updefault}{\color[rgb]{0,0,0}$\gamma$}%
}}}}
\put(406,-2311){\makebox(0,0)[lb]{\smash{{\SetFigFont{9}{10.8}{\familydefault}{\mddefault}{\updefault}{\color[rgb]{0,0,0}$i$}%
}}}}
\put(856,-2311){\makebox(0,0)[lb]{\smash{{\SetFigFont{9}{10.8}{\familydefault}{\mddefault}{\updefault}{\color[rgb]{0,0,0}$j$}%
}}}}
\put(4096,-2311){\makebox(0,0)[lb]{\smash{{\SetFigFont{9}{10.8}{\familydefault}{\mddefault}{\updefault}{\color[rgb]{0,0,0}$k$}%
}}}}
\put(3556,-2311){\makebox(0,0)[lb]{\smash{{\SetFigFont{9}{10.8}{\familydefault}{\mddefault}{\updefault}{\color[rgb]{0,0,0}$i$}%
}}}}
\end{picture}%

%% file: phi-morphisme-1.pstex_t
\begin{picture}(0,0)%
\includegraphics{phi-morphisme-1.pstex}%
\end{picture}%
\setlength{\unitlength}{3947sp}%
\begingroup\makeatletter\ifx\SetFigFont\undefined%
\gdef\SetFigFont#1#2#3#4#5{%
  \reset@font\fontsize{#1}{#2pt}%
  \fontfamily{#3}\fontseries{#4}\fontshape{#5}%
  \selectfont}%
\fi\endgroup%
\begin{picture}(2124,1524)(289,-973)
\end{picture}%

%% file: phi-morphisme-2.pstex_t
\begin{picture}(0,0)%
\includegraphics{phi-morphisme-2.pstex}%
\end{picture}%
\setlength{\unitlength}{4144sp}%
\begingroup\makeatletter\ifx\SetFigFont\undefined%
\gdef\SetFigFont#1#2#3#4#5{%
  \reset@font\fontsize{#1}{#2pt}%
  \fontfamily{#3}\fontseries{#4}\fontshape{#5}%
  \selectfont}%
\fi\endgroup%
\begin{picture}(4310,2882)(273,-2886)
\end{picture}%

%% file: h-morphisme.pstex_t
\begin{picture}(0,0)%
\includegraphics{h-morphisme.pstex}%
\end{picture}%
\setlength{\unitlength}{3108sp}%
\begingroup\makeatletter\ifx\SetFigFont\undefined%
\gdef\SetFigFont#1#2#3#4#5{%
  \reset@font\fontsize{#1}{#2pt}%
  \fontfamily{#3}\fontseries{#4}\fontshape{#5}%
  \selectfont}%
\fi\endgroup%
\begin{picture}(5649,2519)(439,-2108)
\end{picture}%

%% file: stabilisation-1.pstex_t
\begin{picture}(0,0)%
\includegraphics{stabilisation-1.pstex}%
\end{picture}%
\setlength{\unitlength}{3108sp}%
\begingroup\makeatletter\ifx\SetFigFont\undefined%
\gdef\SetFigFont#1#2#3#4#5{%
  \reset@font\fontsize{#1}{#2pt}%
  \fontfamily{#3}\fontseries{#4}\fontshape{#5}%
  \selectfont}%
\fi\endgroup%
\begin{picture}(2274,2911)(439,-2150)
\put(2566,-2086){\makebox(0,0)[lb]{\smash{{\SetFigFont{9}{10.8}{\familydefault}{\mddefault}{\updefault}{\color[rgb]{0,0,0}$n+1$}%
}}}}
\put(2116,-2086){\makebox(0,0)[lb]{\smash{{\SetFigFont{9}{10.8}{\familydefault}{\mddefault}{\updefault}{\color[rgb]{0,0,0}$n$}%
}}}}
\put(1216,-961){\makebox(0,0)[lb]{\smash{{\SetFigFont{9}{10.8}{\familydefault}{\mddefault}{\updefault}{\color[rgb]{0,0,0}$G$}%
}}}}
\put(1441,614){\makebox(0,0)[lb]{\smash{{\SetFigFont{9}{10.8}{\familydefault}{\mddefault}{\updefault}{\color[rgb]{0,0,0}$H$}%
}}}}
\put(1891,-331){\makebox(0,0)[lb]{\smash{{\SetFigFont{9}{10.8}{\familydefault}{\mddefault}{\updefault}{\color[rgb]{0,0,0}$x_0$}%
}}}}
\end{picture}%

%% file: decomposition-polygone.pstex_t
\begin{picture}(0,0)%
\includegraphics{decomposition-polygone.pstex}%
\end{picture}%
\setlength{\unitlength}{3108sp}%
\begingroup\makeatletter\ifx\SetFigFont\undefined%
\gdef\SetFigFont#1#2#3#4#5{%
  \reset@font\fontsize{#1}{#2pt}%
  \fontfamily{#3}\fontseries{#4}\fontshape{#5}%
  \selectfont}%
\fi\endgroup%
\begin{picture}(6114,2152)(1129,-3522)
\put(5851,-2131){\makebox(0,0)[lb]{\smash{{\SetFigFont{9}{10.8}{\familydefault}{\mddefault}{\updefault}{\color[rgb]{0,0,0}4}%
}}}}
\put(6076,-2131){\makebox(0,0)[lb]{\smash{{\SetFigFont{9}{10.8}{\familydefault}{\mddefault}{\updefault}{\color[rgb]{0,0,0}4}%
}}}}
\put(4321,-3211){\makebox(0,0)[lb]{\smash{{\SetFigFont{9}{10.8}{\familydefault}{\mddefault}{\updefault}{\color[rgb]{0,0,0}1}%
}}}}
\put(4321,-2671){\makebox(0,0)[lb]{\smash{{\SetFigFont{9}{10.8}{\familydefault}{\mddefault}{\updefault}{\color[rgb]{0,0,0}1}%
}}}}
\put(4321,-2041){\makebox(0,0)[lb]{\smash{{\SetFigFont{9}{10.8}{\familydefault}{\mddefault}{\updefault}{\color[rgb]{0,0,0}1}%
}}}}
\put(3421,-2401){\makebox(0,0)[lb]{\smash{{\SetFigFont{9}{10.8}{\familydefault}{\mddefault}{\updefault}{\color[rgb]{0,0,0}2}%
}}}}
\put(4951,-2581){\makebox(0,0)[lb]{\smash{{\SetFigFont{9}{10.8}{\familydefault}{\mddefault}{\updefault}{\color[rgb]{0,0,0}3}%
}}}}
\put(4501,-2581){\makebox(0,0)[lb]{\smash{{\SetFigFont{9}{10.8}{\familydefault}{\mddefault}{\updefault}{\color[rgb]{0,0,0}3}%
}}}}
\put(6571,-2176){\makebox(0,0)[lb]{\smash{{\SetFigFont{9}{10.8}{\familydefault}{\mddefault}{\updefault}{\color[rgb]{0,0,0}1}%
}}}}
\put(6571,-1681){\makebox(0,0)[lb]{\smash{{\SetFigFont{9}{10.8}{\familydefault}{\mddefault}{\updefault}{\color[rgb]{0,0,0}1}%
}}}}
\put(6166,-1951){\makebox(0,0)[lb]{\smash{{\SetFigFont{9}{10.8}{\familydefault}{\mddefault}{\updefault}{\color[rgb]{0,0,0}2}%
}}}}
\put(6166,-1636){\makebox(0,0)[lb]{\smash{{\SetFigFont{9}{10.8}{\familydefault}{\mddefault}{\updefault}{\color[rgb]{0,0,0}2}%
}}}}
\put(6166,-2581){\makebox(0,0)[lb]{\smash{{\SetFigFont{9}{10.8}{\familydefault}{\mddefault}{\updefault}{\color[rgb]{0,0,0}2}%
}}}}
\put(6661,-2626){\makebox(0,0)[lb]{\smash{{\SetFigFont{9}{10.8}{\familydefault}{\mddefault}{\updefault}{\color[rgb]{0,0,0}3}%
}}}}
\put(7066,-2626){\makebox(0,0)[lb]{\smash{{\SetFigFont{9}{10.8}{\familydefault}{\mddefault}{\updefault}{\color[rgb]{0,0,0}3}%
}}}}
\put(6571,-3256){\makebox(0,0)[lb]{\smash{{\SetFigFont{9}{10.8}{\familydefault}{\mddefault}{\updefault}{\color[rgb]{0,0,0}1}%
}}}}
\put(1306,-2986){\makebox(0,0)[lb]{\smash{{\SetFigFont{9}{10.8}{\familydefault}{\mddefault}{\updefault}{\color[rgb]{0,0,0}1}%
}}}}
\put(1981,-2986){\makebox(0,0)[lb]{\smash{{\SetFigFont{9}{10.8}{\familydefault}{\mddefault}{\updefault}{\color[rgb]{0,0,0}2}%
}}}}
\put(1306,-2041){\makebox(0,0)[lb]{\smash{{\SetFigFont{9}{10.8}{\familydefault}{\mddefault}{\updefault}{\color[rgb]{0,0,0}3}%
}}}}
\end{picture}%

%% file: link-sign-refinement.tex
\section{Sign assignment induced by the complex}\label{section:sign-assignment}
In this section we prove that the chain complex $\widetilde{C}^- (G)$ coincides with the chain complex $C^- (G)$ over $\zzz$ after a choice of a sign assignment.

\begin{definition}\label{def:assignation-signe}
  A sign assigment is a function $\S:\Rect^{\circ}\rightarrow \{\pm1\}$ such that
\begin{enumerate}
  \item [(Sq)] for any distincts $r_1,r_2,r'_1,r'_2\in \Rect^{\circ}$ such that $r_1 * r_2 = r'_1 * r'_2$ we have
$$\S(r_1).\S(r_2) = - \S(r'_1).\S(r'_2)$$
\item [(V)] if $r_1,r_2\in\Rect^{\circ}$ are such that $r_1 * r_2$ is a vertical annulus then
$$\S(r_1).\S(r_2) = -1$$
\item [(H)] if $r_1,r_2\in\Rect^{\circ}$ are such that $r_1 * r_2$ is a horizontal annulus then
$$\S(r_1).\S(r_2) = +1$$
\end{enumerate}
\end{definition}

To define a sign assignment, we must define a map $s:\ss_n\rightarrow \wss_n$ such that $p\circ s = id_{\ss_n}$ 
$$\xymatrix{
1 \ar[r] & \zzz/2\zzz \ar@{^{(}->}[r]^{i} & \wss_n \ar[r]^p & \ss_n \ar[r] \ar@/^/@{->}[l]^s & 1
}$$

Define $s:\ss_n \rightarrow \wss_n$ by setting:
\begin{itemize}
  \item $s(\mathbf{1})=\wun$ and for all $0\leq i\leq n-2$, $s(\tau_i)=\wtau_i$.

  \item Let $\tau_{i,j}$ be the transposition which exchanges $i$ and $j$ with $j-i>1$.
Then it can be writing in a unique way like that:
$$
\tau_{i,j} = \tau_i . \tau_{i+1} . \,\ldots \,.\tau_{j-2} . \tau_{j-1} .\tau_{j-2} . \, \ldots \,.\tau_{i+1} . \tau_i
$$
We define
$$
s(\tau_{i,j}) = \wtau_i . \wtau_{i+1} . \,\ldots \,.\wtau_{j-2} . \wtau_{j-1} .\wtau_{j-2} . \, \ldots \,.\wtau_{i+1} . \wtau_i
$$

  \item Let $\x\in \ss_n$ and $i_{n-1} \in \{0,\ldots,n-1\}$ be such that $\x(i_{n-1})=n-1$.
Then $\x.\tau_{i_{n-1},n-1} (n-1) = n-1$.
We think of $\x.\tau_{i_{n-1},n-1}$ as an element of $\ss_{n-1}$.
We go on to obtain this kind of writing:
$$
\x = \tau_{i_0 , 0} . \tau_{i_1,1}. \, \ldots \, . \tau_{i_{n-1},n-1}
$$

\begin{rembis}
  In general $x(i_k)\neq k$, except for $k=n-1$.
\end{rembis}

This writing is unique and we let
$$s(\x) = \wtau_{i_0 ,0}.\ldots .\wtau_{i_{n-1},n-1}$$
With this section $s$, every element $\wx \in \wss_n$ has a unique writing called normal form which is
$$\wx = z^u s(\x)$$
where $u\in\{0,1\}$ and $\x = p(\wx)$.
\end{itemize}

To define the sign assignment induced by $\widetilde{C}^- (G)$ we need the 2-cocycle $c\in C^2 (\ss_n ,\zzz/2\zzz)$ associated to the map $s$ given by
\begin{equation}\label{eq:structure algebre}
  s(\x).s(\y) = (i\circ c(\x,\y))s(\x.\y)
\end{equation}
The cohomological class of $c$ measures how $s$ fails to be a group morphism.
In particular, it is non-trivial ($n\geq4$) since $\wss_n$ is a non-trivial central extension of $\ss_n$ by $\zzz/2\zzz$.

We say that a rectangle $r$ is horizontally torn if given the coordonates $(i_{bl},j_{bl})$ of its bottom left corner and $(i_{tr},j_{tr})$ of its top right corner then $i_{bl}>i_{tr}$.
Otherwise, $r$ is said to be not horizontally torn.

\begin{lemma}
The complex $(\widetilde{C}^- (G),\wdm) $ induces a sign assignment in the sense of definition \ref{def:assignation-signe}: for all $(\x,\y)\in \ss_n^2$ and all $r\in \Rect^{\circ}(\x,\y)$
\begin{equation}\label{eq:assignation-signe}
  \S(r) = \varepsilon(r) . c(\x^{-1}.\y,\x)
\end{equation}
where $\varepsilon(r)= +1$ if $r$ is a rectangle not horizontally torn and $\varepsilon(r)= -1$ otherwise.
\end{lemma}

\begin{rembis}
The sign assignment in the sense of definition \ref{def:assignation-signe} is unique up to a 1-coboundary: if $\S_1$ and $\S_2$ are two sign assignments then there exists an application $f:\ss_n \rightarrow \{\pm1\}$ such that for all rectangles $r\in \Rect^{\circ}(\x,\y)$, $\S_1(r) = f(\x).f(\y) .\S_2(r)$.
It is a consequence of the fact that the central extension corresponds to a 2-cohomological class in $H^2 (\ss_n,\zzz/2\zzz)$ (compare with theorem 4.2 \cite{MOST}).
Here, we construct explicitely a map $s:\ss_n\rightarrow \wss_n$ such that $p\circ s = id$ which means making a choice of a representative of this class, another choice must differ by a 1-coboundary.
\end{rembis}

\begin{proof}
Since $c$ is 2-cocycle we have $\delta c =1$ i.e. for all $(\x,\y,\z)\in \ss_n^3$
$$
\delta c (\x,\y,\z) =  c(\y,\z). c(\x.\y , \z) . c(\x ,\y.\z) . c(\x,\y) = 1
$$
By definition we have $ c(\x,\un ) =c(\un,\x) = 1 $ and $ c(\tau_{i,j} , \tau_{i,j} ) =-1 $.
Let's prove that $\S$ satisfy properties (Sq), (V) et (H).
\begin{enumerate}
  \item [(Sq)] Let any four distincts rectangles S $r_1,r_2,r'_1,r'_2\in \Rect^{\circ}$ such that $r_1 * r_2 = r'_1 * r'_2$.
Suppose
$$\wr_1 \in \Rect^{\circ}(\wx,\wx.\wtau_{i,j})$$
corresponds to $r_1$ and
$$\wr_2\in \Rect^{\circ}(\wx.\wtau_{i,j},\wx.\wtau_{i,j}.\wtau_{k,l})$$
corresponds to $r_2$.
Then
$$\wr_1 ' \in \Rect^{\circ}( \wx,\wx.\wtau_{k,l} )$$
corresponds to $r_1 '$ and
$$\wr_2 ' \in \Rect^{\circ}( \wx.\wtau_{k,l},\wx.\wtau_{k,l}.\wtau_{i,j} )$$
corresponds to $r_2 '$.
There are several cases to verify, as for the proof of $\wdm \circ \wdm = 0$ but all cases can be verified in a similar way.
We verify the case $i<j<k<l$.
We calculate $\delta c(\tau_{k,l},\tau_{i,j},\x)$ and $\delta c (\tau_{i,j},\tau_{k,l},\x)$.
With equalities $c (\tau_{i,j}.\tau_{k,l},\x) = c (\tau_{k,l} .\tau_{i,j},\x)$ and $c (\tau_{i,j},\tau_{k,l}) = - c (\tau_{k,l} ,\tau_{i,j})$ we get
$$\S(r_1). \S(r_2 )  = -\S( r'_1 ).\S(  r'_2 )$$

\item [(V)] Let  $r_1,r_2\in\Rect^{\circ}$ such that $r_1 * r_2$ is a vertical annulus.
Suppose that $\wr_1 \in \Rect^{\circ}(\wx,\wx.\wtau_{i})$ corresponds to $r_1$ and $\wr_2\in \Rect^{\circ}(\wx.\wtau_{i},\wx.\wtau_{i}.\wtau_{i,j})$ corresponds to $r_2$.
We calculate $\delta c(\tau_{i},\tau_{i},\x)$ and with equalities $c(\x,\un) =1$, $c(\tau_{i},\tau_{i})=-1$ we get
$$\S( r_1 ).\S (r_2 ) =-1$$

\item [(H)] Let  $r_1,r_2\in\Rect^{\circ}$ such that $r_1 * r_2$ is a horizontal annulus (of height one).
Suppose $\wr_1 \in \Rect^{\circ}(\wx,\wx.\wtau_{i,j})$ corresponds to $r_1$ and $\wr_2\in \Rect^{\circ}(\wx.\wtau_{i,j},\wx.\wtau_{i,j}.\wtau_{j,i})$ corresponds to $r_2$.
We calculate $\delta c(\tau_{i,j},\tau_{i,j},\x)$ and with equalities $c(\x,\un) =1$, $c(\tau_{i,j},\tau_{i,j})=-1$ we get
$$\S( r_1 ).\S (r_2 ) =+1$$
\end{enumerate}
\end{proof}

\begin{proposition}
  The filtered chain complex $(\widetilde{C}^- (G),\wdm) $ is filtered isomorphic to the filtered chain complex $(C^- (G),\partial^- )$.
\end{proposition}

\begin{proof}
The map $s:\ss_n\rightarrow \wss_n$ extends linearly with respect to $\zzz[U_1,\ldots,U_n]$ uniquely to a map $s:C^- (G) \rightarrow \widetilde{C}^- (G)$ which is an isomorphism of modules.
It commutes with the differentials i.e. $s \circ \partial^- = \wdm \circ s$ where the sign assignment $\S$ is given by equation \eqref{eq:assignation-signe}.
By definition, $s$ respects the Alexander filtration and the Maslov grading.
So $s$ defines a filtered isomorphism between the complexes $(C^- (G),\partial^- )$ and $(\widetilde{C}^- (G),\wdm) $.
\end{proof}